\documentclass{amsart}

\usepackage{amsthm, amsmath,latexsym}
\usepackage{amssymb}

\usepackage{setspace}
\usepackage[latin1]{inputenc}
\usepackage{float}

\usepackage{graphicx, graphics}

\usepackage[ruled,vlined]{algorithm2e}
\newtheorem{theorem}{Theorem}[section]
\newtheorem{example}{Example}[section]

\newtheorem{remark}{Remark}[section]

\newcommand{\Q}{\mathbb{Q}}
\newcommand{\Z}{\mathbb{Z}}
\newcommand{\N}{\mathbb{N}}
\newcommand{\R}{\mathbb{R}}

\newcommand{\dsum}{\displaystyle\sum}
\newcommand{\dprod}{\displaystyle\prod}
\newcommand{\dmax}{\displaystyle\max}

\begin{document}

\title{On the complexity of some fuzzy integer programs}
\author{Víctor Blanco and Justo Puerto}
\email{vblanco@us.es; puerto@us.es}

\date{May 15, 2009}

\address{Departamento de Estad\'istica e Investigaci\'on Operativa, Universidad de
Sevilla}



\maketitle

\begin{abstract}
Fuzzy optimization deals with the problem of determining 'optimal'
solutions of an optimization problem when some of the elements that
appear in the problem are not precise. In real situations it is usual to
have information, in systems under consideration, that is not exact. This
imprecision can be modeled in a fuzzy environment. Zadeh
\cite{zadeh65} analyzed systems of logic that permit truth values
between zero and one instead of the classical binary true-false logic. In
this framework, satisfying a certain condition means to evaluate how
close are the elements involved to the complete satisfaction. Then, each
element in a 'fuzzy set' is coupled with a value in $[0,1]$ that represents
the membership level to the set.

In linear programming some or all the elements that describe a problem may be
considered fuzzy: objective function, right-hand side vector or constraint matrix,
the notion of optimality (ordering over the feasible solutions), the level of
satisfaction of the constraints, etc. Moreover, in integer programs, the integrality
constraints may be seen as fuzzy constraints.

Here, we present new complexity results about linear integer
programming where some of its elements are considered fuzzy using
previous results on short generating functions for solving multiobjective
integer programs.
\end{abstract}

\section{Introduction}
Integer linear programming is a special case of linear programming in which all
variables are required to take on integer values only. Its importance is due to the
amount of real-world problems that can be modeled as integer programs. Some
of the main well-known applications of integer programming concerns the
management and efficient use of resources: distribution of goods, production
scheduling, machine sequencing, capital budgeting, etc. Furthermore, from the
mathematical viewpoint many combinatoric and geometrical problems in graph
theory and logic can seen as integer programs. For that reason, and for some
others (interesting research area...), many textbooks and publications are
exclusively devoted to the analysis of integer programs, either from a theoretical
or practical viewpoint. Some of the text books in this area are those by
Schrijver\cite{schrijver86}, Nemhauser-Wolsey \cite{nemhauser-wolsey88},
Sierksma \cite{sierksma02}, among many others.

One of the difficulties when modeling a real-world problem by a mathematical
programming problem is the possible imprecision of the data or the interpretation
of the constraints. For instance, suppose in a simple knapsack problem that we
want to maximize the overall benefits of selecting a subset of items, spending
``around'' \$$b$. We can formulate this problem assuring that we are not going to
spend more than \$$b$ but in the original problem we were determined to spend
around \$$b$ and maybe increasing the budget a little (such as we consider
``around'') one can obtain more benefits. Fuzzy programming deals with this
fuzziness which causes difficulties in modeling.

Linear programming problems can be fuzzified in many ways. For instance,
considering that the coefficients of the objective function or constraints are fuzzy
numbers or introducing vagueness to the inequalities or to the ordering that
determines the maximization/minimization problem. Moreover, in integer
programs, the integrality of the variables may be considered fuzzy just fixing how
close to an integer number one considers that a real number is integer. Some of
these possibilities for mathematical programming problems are described in
\cite{slowinski90} (Chapter 4).

In general, integer linear problems are NP-hard, then the complexity of the
fuzzyfication of these problems is as hard as crisp integer programming.

In this paper we present some complexity results for fuzzy integer problems that
have never been stated before. We give polynomially results in the sense of
Lenstra in his well-known result about the polynomially of (crisp) integer
problems in fixed dimension. Using transformations on fuzzy problems to
multiobjective integer programs we apply generating functions to prove the
polynomial complexity of some models of fuzzy integer problems. Short rational
generating  functions were initially used by Barvinok \cite{barvinok94} as a tool to
develop an algorithm for counting the number of integer points inside convex
polytopes, based in the previous geometrical papers by Brion
\cite{brion88},Khovanskii and Puhlikov \cite{khovanskii-puhlikov92}, and Lawrence
\cite{lawrence91}. The main idea is encoding those integral points in a rational
generating  function in as many variables as the dimension of the space where
the body lives.  Actually, Barvinok presented a polynomial-time algorithm when
the dimension, $n$, is fixed, to compute those functions.

The paper is organized as follows. In Section \ref{sec:sgf} we recall some previous notions and results about short generating functions of rational polytopes for multiobjective integer programming. We present, in Section \ref{sec:fineq}, complexity results for fuzzy integer programs where the inequalities are fuzzified and in Section \ref{sec:fobj} for integer programs with fuzzy coefficients in the objective functions. Finally, in Section \ref{sec:conc} we give some conclusions about the results presented through this paper.

\section{Short rational generating functions}
\label{sec:sgf}
Short rational functions were used by Barvinok \cite{barvinok94}
as a tool to develop an algorithm for counting the number of
integer points inside convex polytopes, based in the previous
geometrical paper by Brion \cite{brion88}. The main idea is
encoding those integral points in a rational function in as many
variables as the dimension of the space where the body lives. Let
$P \subset \R^d$ be a given convex polyhedron, the integral points
may be expressed in a formal sum $f(P,z) = \sum_\alpha z^\alpha$
with $\alpha = (\alpha_1, \ldots, \alpha_d) \in P\cap\Z^d$, where
$z^\alpha = z_1^{\alpha_1}\cdots z_d^{\alpha_d}$. Barvinok's aimed
objective was representing that formal sum of monomials in the
multivariate polynomial ring $\Z[z_1, \ldots, z_n]$, as a
``short'' sum of rational functions in the same variables.
Actually, Barvinok presented a polynomial-time algorithm when the
dimension, $n$, is fixed, to compute those functions. A clear
example is the polytope $P = [0,N] \subset \R$: the long
expression of the generating function is $f(P,z) = \sum_{i=0}^N
z^i$, and it is easy to see that its representation as sum of
rational functions is the well known formula
$\frac{1-z^{N+1}}{1-z}$.

We recall here some results on short rational functions
for rational polytopes, that we use in our development. For details the
interested reader is referred to \cite{barvinok94, barvinok03}.

Let $P = \{x \in \R^n: A\,x \leq b\}$ be a rational polytope in
$\R^n$. The integer points inside $P$ can be encoded in the following ``long" sum of
monomials:
$$
f(P;z) = \dsum_{\alpha\in P\cap\Z^n}\,z^\alpha
$$
where $z^\alpha = z_1^{\alpha_1}\cdots z_n^{\alpha_n}$. Then, to
re-encode, in polynomial-time for fixed dimension, these integer
points in a ``short" sum of rational functions in the form
$$
f(P;z) = \dsum_{i\in I} \varepsilon_i
\dfrac{z^{u_i}}{\dprod_{j=1}^n (1-z^{v_{ij}})}
$$
where $I$ is a polynomial-size indexing set, and where
$\varepsilon \in \{1,-1\}$ and $u_i , v_{ij} \in \Z^n$ for all $i$
and $j$ (Theorem 5.4 in \cite{barvinok94}).

This encoding tool allows us to present algorithms either for counting integer
points inside polytopes (see \cite{barvinok94}) or for solving single (see
\cite{deloera04}) and multi-objective integer problems (see
\cite{blanco-puerto07b, blanco09}). The following result states that encoding the
set of non dominated solutions of a multiobjective integer linear problem can be
done in polynomial time when the dimension is fixed. It will be useful for our
development.
\begin{theorem}[\cite{blanco09}]
\label{theo:nd} Let $A \in \Z^{m\times n}$, $b \in
\Z^m$, $C =(c_1, \ldots, c_k) \in \Z^{k\times n}$, and assume that the number of variables $n$ is
fixed. Suppose $P = \{x \in \R^n : A\,x \leq b, x \geq 0\}$ is a
rational convex polytope in $\R^n$.
Let $\mbox{\it MOILP}_{A,C}(b)$ the following multiobjective linear integer problem
\begin{align}
 \label{MOILP_gral}
\max &\; (c_1\,x, \ldots,  c_k\,x) =: C\,x\tag{$\mbox{\it MOILP}_{A,C}(b)$}\\
s.t.&\; x \in P \cap \Z^n_+
\end{align}
whose solutions (called \textit{nondominated solutions}) are those $x^\star \in P \cap \Z^n_+$ such that there not exist any other $y \in P\cap \Z^n_+$ with $c_i\,x^\star \leq c_i\,y$ for all $i=1, \ldots, k$ with at least one strict inequality.

Then, we can encode, in polynomial time, the entire set of nondominated solutions for
$\mbox{\it MOILP}_{A,C}(b)$ in a short sum of rational functions.
\end{theorem}

One useful result that is used for the proof of the above theorem and for the results presented in this paper is the one that states that the computation of the short generating function of the intersection of two polytopes is doable in polynomial time, for fixed dimension, using the generating functions of both polytopes (Theorem 3.6 in \cite{barvinok03}) . Basically, it uses the Hadamard product of a pair of power series. Given $g_1(z) =
\dsum_{m\in\Z^d}\beta_{m}\,z^m$ and $g_2(z) = \dsum_{m\in\Z^d}\gamma_{m}\,z^m$, the Hadamard product $g = g_1 \ast g_2$ is the power series
$$
g(z) = \dsum_{m \in \Z^n} \eta_m\,z^m \qquad \text{where $\eta_m =
\beta_{m}\gamma_{m}$}.
$$

\section{Integer programs with fuzzy inequalities}
\label{sec:fineq}
Let $P$ be a rational polytope, $A \in \Z^{m\times n}$, $b\in \Z^m$ and $c \in \Z^n$. Consider the following integer program where some constraints are fuzzified:
\begin{equation}
\label{fip}
\begin{array}{lrl}
\max & \;  c\,x& \\
s.t.&  A\,x &\lesssim b\tag{$\mbox{\it FIP}^{\lesssim}_{A,c}(b)$}\\
 & x&\in P \cap \mathbb{Z}_+^n
 \end{array}
 \end{equation}
where $\lesssim$ means that the inequalities must be ``almost satisfied''. The solutions of this problem are then pairs $(x, \mu(x))$ where $x$ is a feasible solution to the problem $\max\{cx : x \in P \cap \Z^n\}$ and $\mu(x)$ the degree of satisfaction of $x$ to the system of inequalities $Ax \leq b$. $\mu$ is called the membership function of the fuzzy set $\tilde{X} = \{(x, \mu(x)): x \in X\}$, where in our case $X=P \cap \Z^n$. The only theoretical requirements to $\mu$ are:
\begin{enumerate}
\item $\mu(X) \subseteq \R_+$.
\item $sup_x \mu(x) < \infty$.
\end{enumerate}
Generally, elements with a zero degree of membership are not listed since it is considered that they are far to be a ``crisp'' (ordinary) element.

If $sup_x \mu(x) =1$ the fuzzy set is called normal, and of course, any non normal fuzzy set can be normalized dividing $\mu$ by $sup_x \mu(x)$. Then, from now on, we consider, w.l.o.g., normal fuzzy sets.

Then, for \ref{fip}, we are interested in membership functions that measure the satisfaction of the constrains or how far is a solution from the crisp system of inequalities, i.e., for each inequality we have a membership function $\mu_i$, $i=1, \ldots, m$ such that:
$$
\mu_i(x) = \left\{ \begin{array}{ll} 0 & \mbox{if $a_i\,x > b_i + \varepsilon_i$}\\
f(x) & \mbox{if $b_i < a_ix \leq b_i + \varepsilon_i$}\\
1 & \mbox{if $a_ix \leq b_i$}
\end{array}\right.
$$
where $\varepsilon_i$ are a nonnegative real numbers that measures how we are considering that a solution does not satisfied the inequalities at all, and $f: \R^n \rightarrow [0,1]$. We assume that $\varepsilon_i \in \Q_+$, for $i=1, \ldots, m$, and that $f(\Z^n) \subseteq \Q \cap [0,1]$. These assumptions are not too restrictive since it is usual to consider linear membership functions (triangular or trapezoidal fuzzy sets, see \cite{zimmermann01}) whose coefficients are rational.

Then, the overall membership function to the system of linear constraints $A\,x \leq b$ is given by:
$$
\mu(x) = \min_i \mu_i(x).
$$

We are interested in maximizing the linear function $c\,x$ in \ref{fip} such that $x \in P \cap \Z^n$ and with the maximum value of the membership function. Since on $R^2$ the componentwise ordering is weak, the above problem is equivalent to finding the nondominated solutions to a biobjective problem with objective functions $(c\,x, \mu(x))$. Then, we deal with the following equivalent biobjective crisp integer problem:
\begin{equation}
\label{crisp0}
\begin{array}{lr}
\max & \;  (c\,x, \mu(x)) \\
s.t. & x\in P \cap \mathbb{Z}_+^n
 \end{array}
 \end{equation}
where the fuzzy constraints have been substituted by the membership function $\mu$ in the objective functions.

By definition of $\mu$, Problem \ref{crisp0} is equivalent to:
\begin{equation}
\label{crisp1}
\begin{array}{lr}
\max & \;  (c\,x, \min_i \mu_i(x)) \\
s.t. & x\in P \cap \mathbb{Z}_+^n
 \end{array}
  \end{equation}
 and then, equivalent to the following biobjective mixed integer program:
 \begin{equation}
\label{crisp2}
\begin{array}{lrl}
\max & \;  (c\,x, z) \\
s.t. & z \leq \mu_i(x),& i=1, \ldots, m\\
 &x\in P \cap \mathbb{Z}_+^n\\
 &z \in [0, 1]
 \end{array}
\end{equation}
Let us consider $\mu_i$ to be rational linear functions for all $i=1, \ldots, m$. Then, because the number of feasible solutions in the $x$ variable is finite ( $x \in \Z^n$ and $P$ is a polytope) and for each feasible solution $\overline{x}$, $z$ is fixed as $\min_i  \mu_i(\overline{x})$, the number of possible nondominated solutions of Problem \ref{crisp2} is finite. Furthermore, $z$ is rational because $x\in \Z^n$ and $\mu_i$ are rational linear functions. Then, the variable $z$ can be transformed to an integer variable $y=Mz$, with $M$ the least common multiple of all the denominators that appear in the inverses of all subdeterminants of the matrix defined by the linear functions $\mu_i(x)$, $i=1, \ldots, m$. Replacing each inequality $z \leq \mu_i(x)$ by $y \leq M\mu_i(x)$, we have that the equivalent problem:
 \begin{equation}
\label{crisp3}
\begin{array}{lrl}
\max & \;  (c\,x, y) \\
s.t. & y \leq M\mu_i(x),& i=1, \ldots, m\\
 &x\in P \cap \mathbb{Z}_+^n\\
 &y \in [0, M]\cap \Z
 \end{array}
\end{equation}
is a biobjective integer linear program that has the same set of solutions but  where the solutions in $y$ are related with those in $z$ dividing by  $M$. Furthermore, if $c$ is generic for the problem $\max\{cx: x \in P \cap \Z^n\}$, Problem \ref{crisp3} has at most $M+1$ nondominated solutions since for each value of $y$ there is exactly one solution in $x$. Then, this problem may be solved solving $M+1$ single objective integer problems, one for each of the possible values of $y$.
%

It is worth noting that fuzzy integer programming is NP-hard. Indeed, reduction comes from crisp integer programming. It is well-known that finding an optimal solution of a general integer program, when the dimension is part of the input, is NP-hard (see \cite{schrijver86}). Thus, since we can state that $\hat{x}$ is a an optimal solution to Problem \eqref{crisp0} if and only if there exists $\hat{y} \in [0,M] \cap \Z$ such that $(\hat{x}, \frac{\hat{y}}{M})$ is an optimal solution to the fuzzy integer program \ref{fip}, the conclusion follows.

This shows that fuzzy integer programming is as hard as crisp integer
programming. Nevertheless, the situation is even harder because crisp integer
programming in fixed dimension is polynomial (see \cite{lenstra81}) but fuzzy
integer programming is equivalent to bicriteria linear integer programming (see
\cite{blanco09}) which is also NP-hard.

In spite of that, there is a natural (not easy) way to find all the solutions to \ref{fip} which is based on solving $M+1$ crisp integer problems of the form \eqref{crisp0}, fixing $y=0, 1, \ldots, M$. However this approach does not ensure polynomiality even in fixed dimension. Here the problem comes from $M+1$, the number of problems to be solved. This figure might be exponential in the input size and therefore even solving each subproblem in polynomial time the overall complexity will be only pseudopolynomial.

The best complexity result that we can state is given by the next theorem.
\begin{theorem}
\label{theo:complex1}
We can encode in polynomial time, for fixed dimension, the entire set of solutions for
\ref{fip} in a short sum of rational functions.
\end{theorem}
\begin{proof}
Using Barvinok's algorithm (Theorem 5.4 in \cite{barvinok94}),
compute the following generating function in $2n$ variables:
\begin{equation}
\label{rf_tf} f(x_1, x_2) := \dsum_{((u, y_u), (v, y_v)) \in P_{C} \cap \Z^{2n}}
x_1^{(u, y_u)}\,x_2^{(v, y_v)}
\end{equation}
where $\widetilde{P} = \{((u,y_u), (v, y_v)) \in \Z^{n+1} \times \Z^{n+1} : u, v \in P, c\,u -
c\,v \geq 0,  y_u  \leq y_v \text{ and } c\,u + y_u - c\,v - y_v \geq 1\}$. $\widetilde{P}$ is clearly a
rational polytope. For fixed $u \in \Z^n$, the second components,
$(v, y_v)$, in the monomial $x_1^{(u, y_u)}\,x_2^{(v, y_v)}$ of $f(x_1,x_2)$ represent
the solutions dominated by $(u, y_u)$.

 Now, for any function $\varphi$, let $\pi_{1, \varphi}, \pi_{2, \varphi}$
  be the projections of $\varphi(x_1, x_2)$ onto the $x_1$- and $x_2$-variables, respectively. Thus $\pi_{2,f}(x_2)$
 encodes all dominated feasible integral vectors (because the degree vectors of
the $x_1$-variables dominate them, by construction), and it can be
computed from $f(x_1, x_2)$ in polynomial time by Theorem 1.7 in
\cite{barvinok94}.

Let $V(P)$ be the set of extreme points of the polytope $P$ and
choose an integer $R \geq \max\{v_i: v\in V(P), i=1,\ldots,n\}$
(we can find such an integer $R$ via linear programming). For this
positive integer, $R$, and $M$ as described above, let $r((x,z), (R, M))$ be the rational function for
the polytope $\{(u, y_u) \in \R^{n+1}_+: u_i\leq R, i=1, \ldots, n, \text{ and } y_u \leq M+1\}$, its expression is:
$$
r((x,z),(R,M))= \left( \dfrac{1}{1-z} +
\dfrac{z^{M+1}}{1-z^{-1}}\right)\dprod_{i=1}^n \left( \dfrac{1}{1-x_i} +
\dfrac{x_i^R}{1-x_i^{-1}}\right).
$$

Define $f(x_1, x_2)$ as above, $\pi_{2,f}(x_1)$ the projection of $f$
onto the second set of variables as a function of the
$x_1$-variables and $F(x_1)$ the short generating function of $P$.
They are computed in polynomial time by Theorem 1.7 and Theorem
5.4 in \cite{barvinok94} respectively. Compute the following
difference:
$$
h(x_1) := F(x_1) - \pi_{2,f}(x_1).
$$

This is the sum over all monomials $x_1^{(u, y_u)}$ where $(u, y_u) \in P\times [0, M+1]$ is a
nondominated solution, since we are deleting, from the total sum
of feasible solutions, the set of dominated ones.

This construction gives us a short rational function associated
with the sum over all monomials with degrees being the
nondominated solutions for Problem \ref{crisp3}. As a consequence, we
can compute the number of nondominated solutions for the problem.
The complexity of the entire construction being polynomial since
we only use polynomial time operations among four short rational
functions of polytopes (these operations are the computation of
the short rational expressions for $f(x_1,x_2)$, $F(x_1)$, $r((x, z),(R, M))$ and
$\pi_{2,f}(x_1)$).
\end{proof}

The above result states that the solution of the fuzzy problem can be encoded in a short rational generating function in polynomial time for fixed dimension. However, to obtain the explicit list of solutions we should expand, as a Laurent series, the rational functions that appear in that expression.

In the following, we present an efficient procedure to obtain the entire set of solutions for \ref{fip}. For that, we concentrate on a different concept of complexity that has been already used in the literature for
slightly different problems. Computing maximal independent sets on graphs is known to be $\#$P-hard (\cite{garey-johnson79}),
nevertheless there exist algorithms for obtaining these sets which
ensure that the number of operations necessary to obtain two
consecutive solutions of the problem is bounded by a polynomial in
the problem input size. These
algorithms are called polynomial delay. Formally, an algorithm is
said \emph{polynomial delay} if the delay, which is the maximum
computation time between two consecutive outputs, is bounded by a
polynomial in the input size (\cite{johnson88}).

In our case, a polynomial delay algorithm, in fixed dimension, for
solving \ref{fip} means that once the first solution is computed, either in polynomial
time a next fuzzy solution is found or the termination of the algorithm is given as an output.

Next, we present a polynomial delay algorithm, in fixed dimension,
for solving \ref{fip}.

Let $L = \max\{U, l^{-1}\}$ where $U$ and $l$ are respectively, the largest and smallest element that appear in the description of $P$ as a system of inequalities. The pseudocode of a procedure for obtaining the set of solutions of \ref{fip} is shown in Algorithm \ref{alg:logM}.

{\sffamily
\begin{algorithm}[!h]
\label{alg:logM} \SetLine \SetKwInOut{Input}{input}
\SetKwInOut{Output}{output} \textbf{Initialization: }
$\mathcal{M}=[0,L]^n \subseteq P$.

\textbf{Step 1: } Let $\mathcal{M}_1, \ldots, \mathcal{M}_{2^n}$
be the hypercubes obtained dividing $\mathcal{M}$ by the midpoints of its edges.

$i=1$

 \textbf{Step 2: }

\Repeat{$i<=2^n$}{Count the elements encoded in
$r_{\mathcal{M}_i}(x)\ast h(x)$: $n_{\mathcal{M}_i}$. This is the
number of nondominated solutions in the hypercube $\mathcal{M}_i$.

\eIf{$n_{\mathcal{M}_i}=0$}{ \eIf{$i<2^n$}{$i\leftarrow i+1$}{Go
to Step 1 with $\mathcal{M}$ the next hypercube to its predecessor
hypercube}}{

\eIf{$n_{\mathcal{M}_i}=1$}{Let $x^*$ the unique solution in
$\mathcal{M}_i$, $ND = ND \cup \{x^*\}$ and $i\leftarrow i+1$ }{
Go to Step 1 with $\mathcal{M}=\mathcal{M}_i$}}}

\Output{ND}
\caption{Binary search algorithm for solving MOILP using SGF.}
\end{algorithm}}
 The following result states the complexity of this algorithm.
\begin{theorem}
\label{theo:complexfip2}
If the dimension is fixed, Algorithm \ref{alg:logM} is a polynomial-delay method to find all the solutions of \ref{fip}.
\end{theorem}
\begin{proof}
Let consider the multiobjective transformation of \ref{fip} in \eqref{crisp3}.

By definition, $P \subseteq  [0,L]^n$. Let $h(x)$ denote the short generating function encoding the nondominated solutions of \eqref{crisp3}, $r_\mathcal{H}(x)=\dprod_{i=1}^n\big[\dfrac{x_i^{m_i}}{1-x_i} +
\dfrac{x_i^{M_i}}{1-x_i^{-1}}\big]$ be the short generating function of the hypercube $\mathcal{H}
= \dprod_{i=1}^n [m_i, M_i] \subseteq \R^n$, with $m_i, M_i \in
\Q$ for $i=1, \ldots, n$ and $\texttt{counting\underline{ }integer(}\mathcal{H}\texttt{)}$ denote the number of operations needed to count the integer points encoded in $r_\mathcal{H}\ast h$ (this number is polynomially bounded, when the dimension is fixed by Theorem 1.2 in \cite{barvinok94}).

The algorithm proceeds on a recursive subdivision of hypercubes (subhypercubes). Starting from the original hypercube $[0,L]^n$, we subdivide it in $2^n$ subhypercubes testing whether they contain nondominated solutions (this process is done using the short generating function of the corresponding hypercube, $r_\mathcal{H}$, and $h(x)$).

The subhypercubes that do not contain nondominated solutions are discarded from consideration.

Each subhypercube that contains at least one nondominated solution is subdivided further until we are led to a family of elements in this subdivision so that each one contains either a unique nondominated solution or it does not contain any nondominated solution (this search is done using a depth-first search, see \cite{swamy81}). Each nondominated solution is added to $ND$, our current set of solutions. Then, all this family of subhypercubes is removed from further consideration.

The algorithm repeats this scheme with all the elements in the subdivision until all of them have been processed: either fathomed or considered because they contain a nondominated solution.

The entire process is polynomial delay because we are searching in a binary tree with at most $nlog(L)$ levels (due to the binary subdivision of the hypercube $[0,L]^n$). At each node the algorithm checks in polynomial time, whether the branches at this node contain at least one nondominated solution using the test given by $\texttt{counting\underline{ }integer(}\mathcal{H}\texttt{)}$. Therefore, the algorithm only processes those branches where it is ensured a final success (i.e. a nondominated solution will be found). Hence, from the last solution found, the overall number of operations until a next solution is found or finding a certificate of termination is bounded above by $O(nlog(L)) \times \texttt{counting\underline{ }integer(}\mathcal{H}\texttt{)}$.
\end{proof}
\begin{remark}
One may think of using a different approach to enumerate the entire set of
solutions of \ref{fip} that may lead to a simpler polynomial delay algorithm. The
idea would be to use the equivalence between \ref{fip} and solving a series of
$M+1$ crisp IP (see the equivalence above). In fix dimension, it is known that
solving each of these problems is polynomially doable \cite{lenstra81}. Then from
one solution to the next one, generated in this way, the method would need a
polynomial number of operations and therefore there would be, at most, a
polynomial delay between two consecutive solutions found.  However, this simple
method does not guarantee the complete enumeration of the set of solution of
\ref{fip} since each of these IP may have multiple optima. The reader may note
that to ensure the entire enumeration of the set of optimal solutions the method
would need to find \textit{all}  the alternative optima of each integer problem and
this process is equivalent to enumeration of integer points in polyhedra  which
would lead us again to the starting point.
\end{remark}
The following example illustrates a strategy to parametrically fix the membership function of the inequalities in \ref{fip}.
\begin{example}
\label{ex:hyperplanes}
Let $A \in \Z^{m\times n}$ (with rows $a_1, \ldots, a_m$) and $b = (b_1, \ldots, b_m) \in \Z^m$ and
$$
f_i(x) = \dfrac{p_i}{q_i} \left(b_i - \dsum_{j=1}^n a_{ij}\,x_j\right)
$$
for $x=(x_1, \ldots, x_n)$ and $p_i, q_i \in \Z_+$ for $i=1, \ldots, m$.


Then, the hyperplane of $\R^{n+1}$ given by $x_{n+1}-f_i(x)=0$ is the hyperplane that contains the $i$-th face of $\{x \in \R^{n}: Ax \leq b\}$ (embedded in the hyperplane of $x_{n+1}=0$) and that forms an angle $\arcsin\left(\sqrt{\frac{p_i}{q_i}\sqrt{\dsum_{i=1}^n a_{ij}^2}}\right) \in [0, \frac{\pi}{2}]$ with the hyperplane $x_{n+1}=0$.

Defining $f(x) = \min_i f_i(x)$ we have that for each $x \in \R^n$:
\begin{enumerate}
\item If $Ax \leq b$, then $f(x)\geq 0$. And $f(x)=0$ if and only if $Ax=b$.
\item If there exists any $i \in \{1,  \ldots, m\}$ such that $a_ix >b$, then $f(x)<0$.
\end{enumerate}

Then, defining
$$
\mu(x) = \left\{\begin{array}{ll}
1 & \mbox{if $f(x)\leq 0$}\\
1+f(x) & \mbox{if $-1 \leq f(x) < 0$}\\
0 & \mbox{if $f(x) <-1$}
\end{array}\right.
$$

clearly $\mu$ is a normal membership function modeling the satisfaction of the system of inequalities $Ax \leq b$.

In order to minimize as in Problem \ref{crisp0}, we can assume that $\mu(x)$ is basically $1+f(x)$, with constraints $-1\leq f(x) \leq 0$ since in the points inside the polytope $\{x \in \R^n: Ax \leq b\}$ are catched by those that maximize $1+f(x)$ (or equivalently, those that maximize $f(x)$) and we are not interested in those points that have $\mu(x)=0$, so we can obviate the third case in the definition of $\mu(x)$.
\end{example}

\begin{example}[\cite{herrera-verdegay95}]
\label{ex:1}
Consider the following problem:
\begin{equation}
\label{example1}
\begin{array}{lrl}
\max & \;  2\,x_1+5\,x_2& \\
s.t.&  2x_1 -x_2 &\lesssim 9\\
&  2x_1 + 8x_2 &\lesssim 31\\
 & x_1, x_2&\in \mathbb{Z}_+
 \end{array}
 \end{equation}
 We use the membership functions proposed in Example \ref{ex:hyperplanes} with $p_1=p_2=1$ and $q_1=3$, $q_2=4$ (using the same membership functions as in \cite{herrera-verdegay95}). Figure \ref{fig:example} shows the crisp polytope and the area where the membership function is not zero.
 \begin{figure}[H]
 \begin{center}
 \includegraphics[scale=0.5]{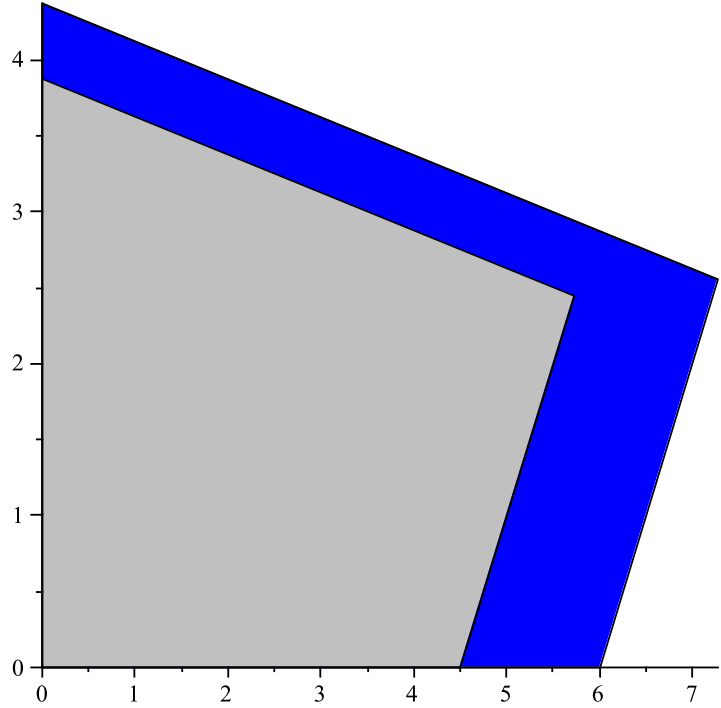}
 \caption{The crisp polytope ($P$)  and its maximum deformation by the membership function ($\widetilde{P}$) of Example \ref{ex:1}\label{fig:example}.}
 \end{center}
 \end{figure}
  Transforming the problem as in \eqref{crisp2} we obtain the following biobjective mixed-integer problem:
 \begin{equation}
\label{ex:crisp2}
\begin{array}{lrl}
\max & \;  (2\,x_1+5\,x_2, z) \\
s.t. & z \leq \frac{12-2x_1+x_2}{3}&\\
 & z \leq \frac{35-2x_1-8x_2}{4}&\\
 &x_1, x_2\in \mathbb{Z}_+\\
 &z \in [0, 1]
 \end{array}
\end{equation}
Figure \ref{fig:3d} shows the feasible region of the above biobjective problem. The bootom of that polytope ($z=0$) coincides with the embedding of the crisp original polytope and in the top ($z=1$) appears the maximum deformation of the crisp polytope by the fuzzyfication.
\begin{figure}[H]
\begin{center}
\includegraphics[scale=0.6]{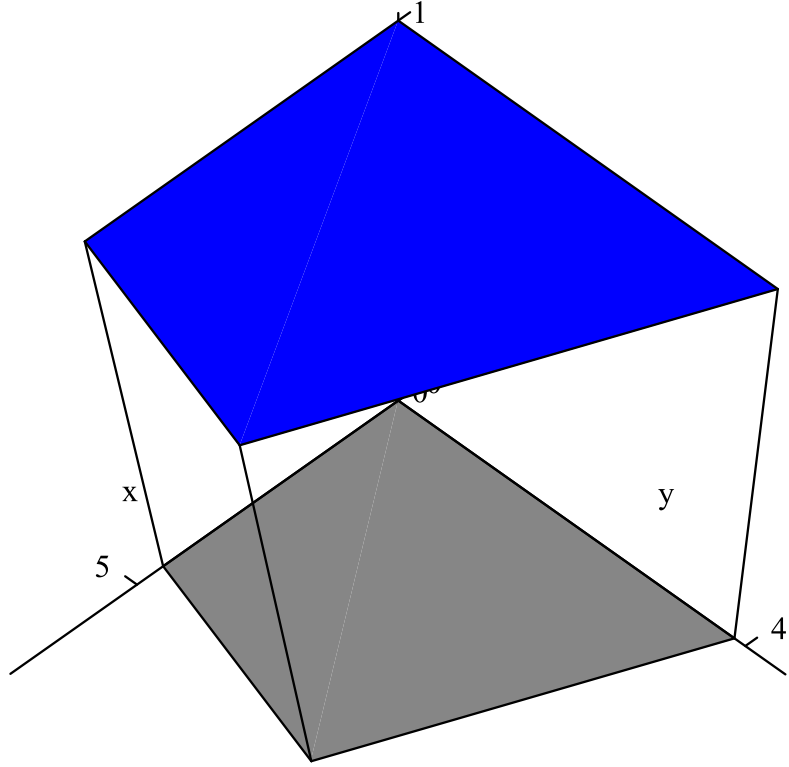}
\end{center}
\caption{Feasible polytope of Problem \ref{ex:crisp2}\label{fig:3d}}
\end{figure}
Now, the least common multiple of all the denominators that appear in the inverses of all subdeterminants of the matrix defined by the membership functions is $M=12$, then, the problem above is equivalent to the biobjective integer problem:
 \begin{equation}
\label{ex:crisp3}
\begin{array}{lrl}
\max & \;  f(x_1, x_2, y):= (2\,x_1+5\,x_2, y) \\
s.t. & y \leq 4\,(12-2x_1+x_2)&\\
 & y \leq 3\,(35-2x_1-8x_2)&\\
 &x_1, x_2\in \mathbb{Z}_+\\
 &y \in [0, 12]\cap \Z
 \end{array}
\end{equation}
The nondominated solutions of Problem \eqref{ex:crisp3} are $x_1^*=5, x_2^*=3, y^* =3$ ($f^*=(25, 3)$), $x_1^*=4, x_2^*=3, y^*=9$ ($f^*=(23, 9)$) and $x_1^*=3, x_2^*=3, y^*=12$ ($f^*=(21,12)$). Then, the set of solutions of the fuzzy problem, with its respective membership values (obtained dividing $y$ by $M$) are:
$$
\{[(5,3)/0.25], [(4,3)/0.75], [(3,3)/1]\}
$$
\end{example}

\section{Fuzzy integer programs with fuzzy objective coefficients}
\label{sec:fobj}
In this section we deal with linear integer problems where the coefficients of the objective functions are fuzzy numbers:
\begin{equation}
\label{fip2}
\begin{array}{lrl}
\max & \;  \tilde{c}\,x& \\
s.t.& x\in P \cap \mathbb{Z}_+^n \i\tag{$\mbox{\it FIP}_{A,\tilde{c}}(b)$}
 \end{array}
 \end{equation}
 where $P$ is a rational polytope and $\tilde{c} = (\tilde{c}_1, \ldots, \tilde{c}_n)$ is a vector of fuzzy numbers.

First of all, we describe an equivalent way to handle fuzzy numbers that will be useful to tackle \ref{fip2}. Let $\tilde{c}$ be a real fuzzy number, for each $\alpha \in [0,1]$ the $\alpha$-cut of $\tilde{c}$ is the set
$$
\tilde{c}_\alpha = \{ c \in \R: \mu_{\tilde{c}} \geq \alpha\}.
$$
It is clear that $\tilde{c}$ is totally determined by its set of $\alpha$-cuts for $\alpha\in[0,1]$ since the membership function that determines $\tilde{c}$ can be identified with this family of sets. Actually, the expression for the membership function, given the set of $\alpha$-cuts is:
$$
\mu_{\tilde{c}}(x) = \sup_{\alpha \in (0,1]} \min\{\alpha, \chi_{\tilde{c}_\alpha}(x)\}
$$
where $\chi_{\tilde{c}_\alpha}$ is the characteristic function of the $\alpha$-cut of $\tilde{c}$.

Furthermore, for each $\alpha \in (0,1]$, $\tilde{c}_\alpha$ is a closed interval in $\R$:
$$
\tilde{c}_\alpha = [c_1^\alpha, c_2^\alpha] \quad \alpha \in (0,1]
$$
Although in general, it is necessary to have the complete set of intervals that describes the $\alpha$-cuts, in many well-known families of fuzzy numbers, we can describe the fuzzy number just giving a finite subset of $\alpha$-cuts. For instance, interval fuzzy numbers are totally described using one $\alpha$-cut and triangular or trapezoidal using exactly two $\alpha$-cuts.

The objective function of our problem, \ref{fip2}, is $\tilde{c}x = \dsum_{i=1}^n \tilde{c}_i x_i$ that is a fuzzy number. Let $[c_{i1}^\alpha, c_{i2}^\alpha]$ the $\alpha$-cut for $\tilde{c}_i$, $\alpha\in(0,1]$, $i=1, \ldots, n$. Then, using the addition and multiplication of a fuzzy number by an ordinary number (see   \cite{kaufmann-gupta91}) , the $\alpha$-cut for $\tilde{c}x$ is given by:
  $$
  (\tilde{c} x)_\alpha = [\dsum_{i=1}^n c_{i1}^\alpha\,x_i, \dsum_{i=1}^n c_{i2}^\alpha\,x_i]
  $$

Now, evaluating a feasible solution $x$ in the objective function means to compute a fuzzy number or equivalently, its $\alpha$-cuts. To compare two feasible solutions $x$ and $y$, we have to compare fuzzy numbers (that are actually functions), but orderings defined over functions are not total, but partial (it may exists functions that are not comparable). However, by means of $\alpha$ cuts, we can compare both fuzzy numbers (the evaluation by the objective function in $x$ and $y$) using the equivalent way to treat them.

Let $\tilde{c}$ an $\tilde{d}$ be two real fuzzy numbers, we say that $\tilde{c} \leq \tilde{d}$ if $c_1^\alpha \leq d_1^\alpha$ and $c_2^\alpha \leq d_2^\alpha$ for all $\alpha \in (0,1]$. Then, to compare two fuzzy numbers by this partial ordering we just need the extreme points of each of its $\alpha$-cuts intervals, i.e., to compare $\tilde{c}$ and $\tilde{d}$ we only need to compare by the componentwise order in $\R^2$ the set of vectors $\{(c_1^\alpha, c_2^\alpha): \alpha \in (0,1]\}$ and $\{(d_1^\alpha, d_2^\alpha): \alpha \in (0,1]\}$ for each $\alpha$-cut.

For our particular case, for two feasible solutions, $x$ and $y$, we need to compare $\dsum_{i=1}^n c_{i1}^\alpha\,x_i$ and $\dsum_{i=1}^n c_{i1}^\alpha\,y_i$, and $\dsum_{i=1}^n c_{i2}^\alpha\,x_i$ and $\dsum_{i=1}^n c_{i2}^\alpha\,y_i$.

Then, we can transform \ref{fip2} to a continuum family of biobjective problems:
\begin{equation}
\label{tfip2}
\begin{array}{lrl}
\max & \;  (c_1^\alpha\,x, c_2^\alpha\,x)& \\
s.t.& x\in P \cap \mathbb{Z}_+^n \i\tag{$\mbox{\it FIP}^\alpha_{A,\tilde{c}}(b)$}
 \end{array}
 \end{equation}
 for each $\alpha \in (0,1]$ and where $c_1^\alpha= (c_{11}^\alpha, \ldots, c_{1n}^\alpha)$ and $c_2^\alpha= (c_{21}^\alpha, \ldots, c_{2n}^\alpha)$ are the lower and upper extremes of the $\alpha$-cuts of $\tilde{c}$, respectively.

 The set of solutions of these problems are the points $x \in P\cap \Z^n_+$ such that there is no $y \in P \cap \Z^n_+$ with $c_1^\alpha x \leq c_1^\alpha y$ and $c_2^\alpha x \leq c_2^\alpha y$ for all $\alpha \in (0,1]$.

In \cite{gonzalez-vila91} the authors propose a way to reduce this continuum (in $\alpha$) family of problems to a discrete one. Let us consider $\alpha_1, \ldots, \alpha_k \in (0,1]$, that we call \textit{ranking system}. Then, we can assume that instead of considering the entire interval $(0,1]$ we consider only a representative set of elements of this interval. In practice, many problems are completely determined by a finite subset of $\alpha$-cuts, so we are not loosing information assuming this ``discretization''. In those cases, where all the $\alpha$-cuts are needed, we can consider an approximation to the corresponding fuzzy numbers with as many elements (finite) in the ranking system as we want.

The transformation above is exact at the nodes $\alpha_i$ of the representation and it has some global errors on $[0, 1]$. Further, it is easy to control the error by introducing additional nodes into the representation or by using a sufficiently high number
of nodes with $\dmax_i\{\alpha_i-\alpha_{i-1}\}$ sufficiently small. To control the error of the
approximation, we can proceed by increasing the number $k + 1$ of elements in the ranking system; a
possible strategy is to double the number of points by using $k = 2^s$ and by
moving automatically to $k = 2^{s+1}$ if a better precision is necessary.

In the standard models, a finite and pre-specified ranking system describes exactly the fuzzy numbers. For instance, trapezoidal fuzzy numbers (including triangular fuzzy number as an special case) is totally characterized by two elements in the ranking system $\{\alpha_0, 1\}$. In general it is also true for piecewise linear fuzzy numbers where the ranking system is $\{\alpha_1, \ldots, \alpha_k=1\}$, being $\alpha_i$ each one of the vertices of the polygonal that gives the membership function. This ranking system describes completely the fuzzy number.

The following example illustrates this idea.
\begin{example}
Let $\tilde{x}$ be a trapezoidal fuzzy number with membership function given by:
$$
\mu_{\tilde{x}}(z) = \left\{\begin{array}{ll} 0 & \mbox{if $z < a_1$}\\
\frac{z-a_1}{a_2-a_1} & \mbox{if $a_1\leq z \leq a_2$}\\
1 & \mbox{if $a_2\leq z \leq a_3$}\\
\frac{a_4-z}{a_4-a_3} & \mbox{if $a_3\leq z \leq a_4$}\\
 0 & \mbox{if $x > a_4$}
 \end{array}\right.
 $$
\begin{figure}[H]
 \begin{center}
\includegraphics[width=7cm, height=5.5cm]{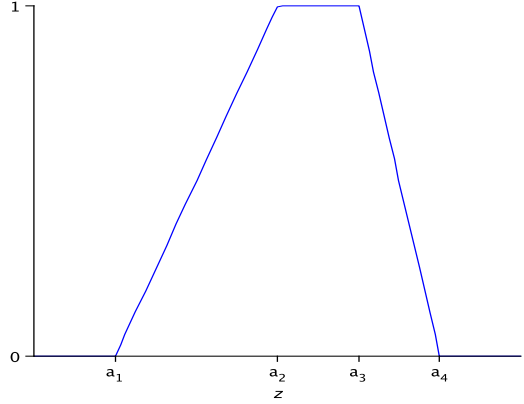}
\caption{Trapezoidal membership function.\label{fig:trapez}}
\end{center}
\end{figure}
Figure  \ref{fig:trapez} shows the membership function of a trapezoidal fuzzy number. Figure \ref{fig:trapez2} shows the $\alpha_0$-cut and the $1$-cut for some $\alpha_0 \in (0,1)$. From the $1$-cut, the elements $a_2$ and $a_3$ of the fuzzy number are determined. From the $\alpha_0$-cut, $[l_1, l_2]$, the equation of the line that pass through the points $(l_1, \alpha_0)$ and $(a_2, 1)$ intersects with the $x$-axis in $(a_1,0)$ and the line that pass through $(l2, \alpha_0)$ and $(a_3, 1)$ intersects with the $x$-axis in $(a_4, 0)$. Then, we have completely determined the fuzzy number.
\begin{figure}[H]
\begin{center}
\begin{tabular}{cc}
\includegraphics[width=6cm, height=5cm]{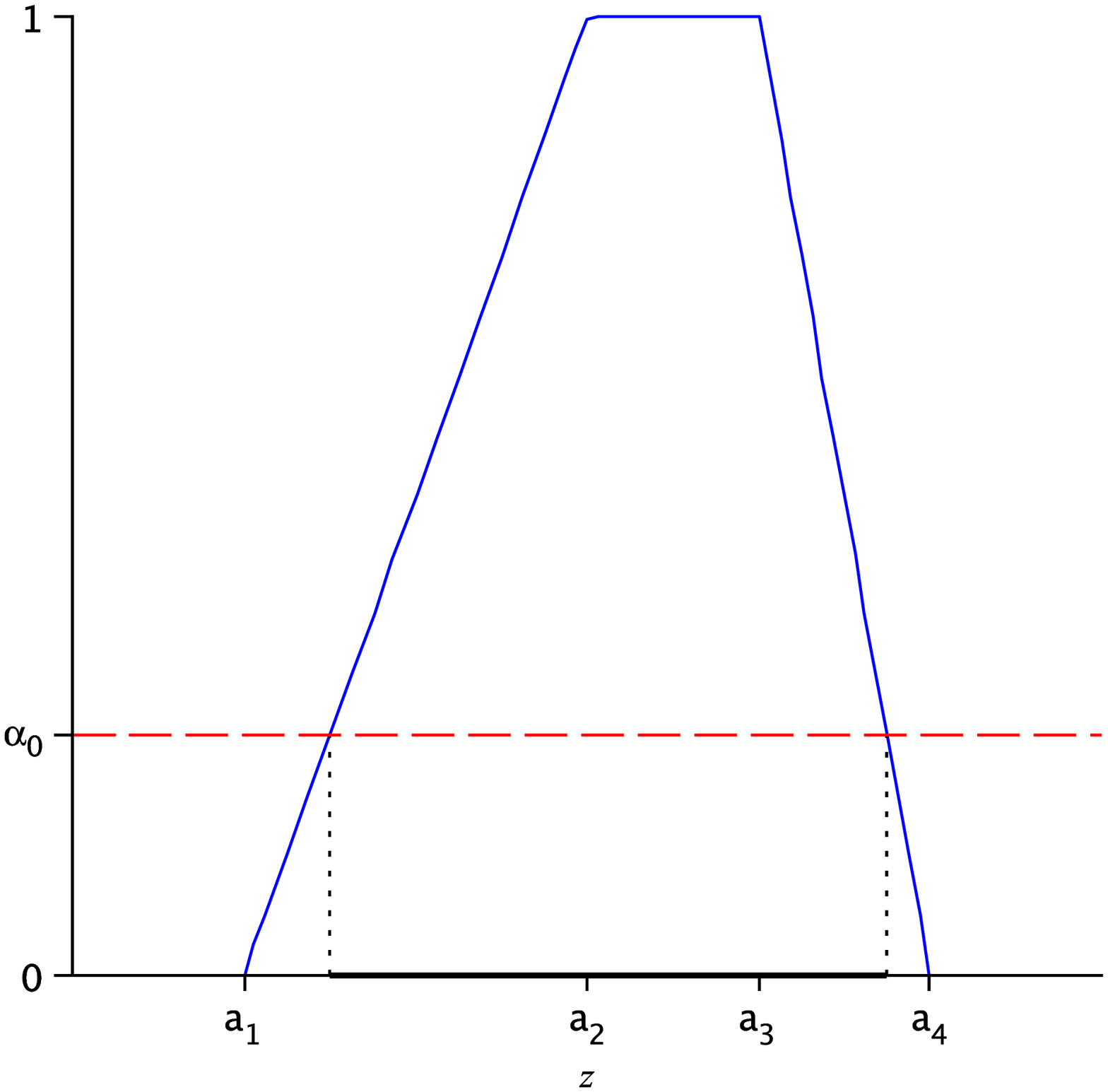} & \includegraphics[width=6cm, height=5cm]{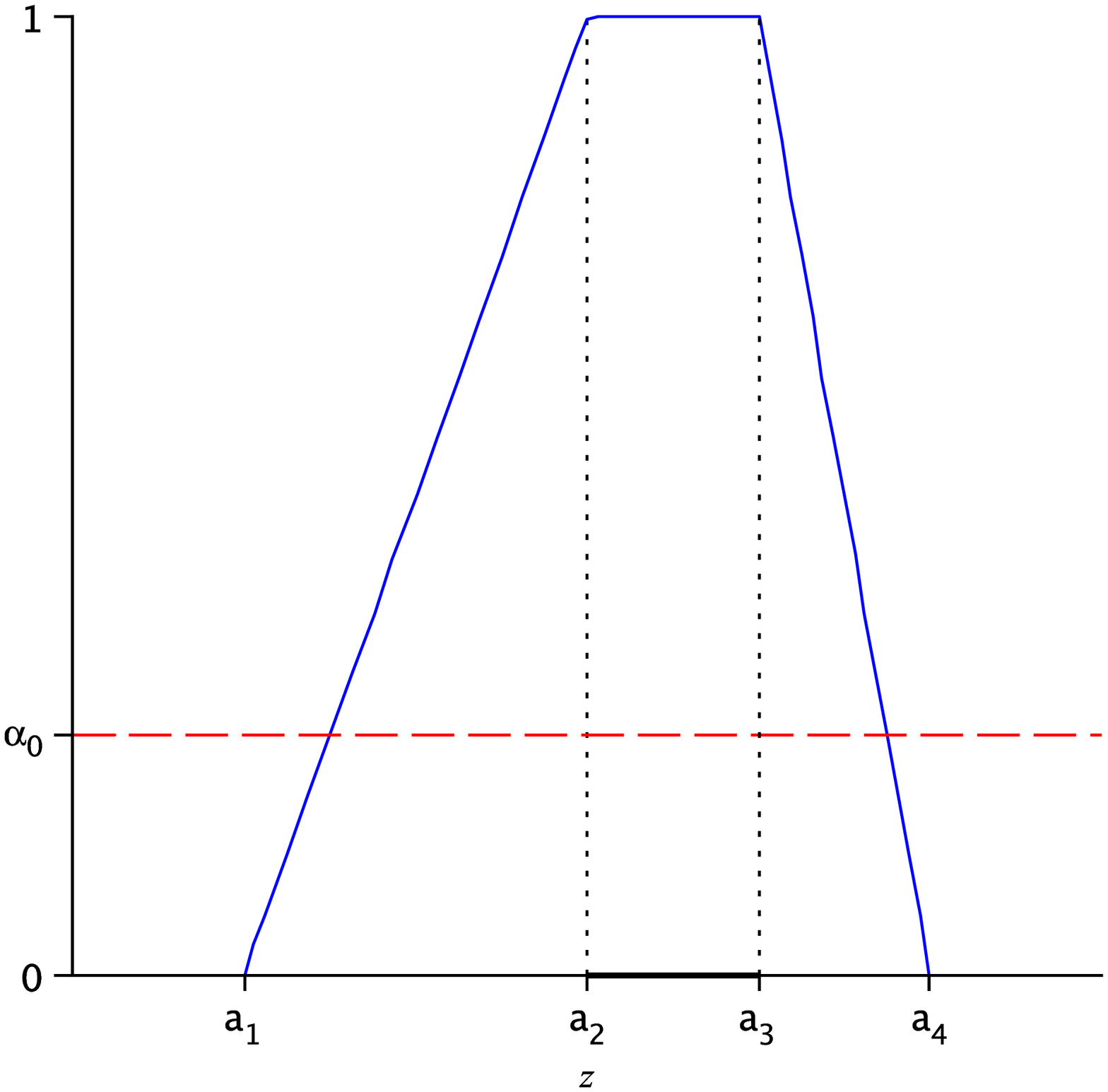}
\end{tabular}
\caption{$\alpha_0$-cut and $1$-cut for a trapezoidal fuzzy number.\label{fig:trapez2}}
\end{center}
\end{figure}
\end{example}
The following result states the complexity of encoding the solutions of \ref{fip2} in a short generating functions when the membership functions involved in the problem are piecewise linear functions (or equivalently, that the problem is totally described by a finite ranking system).
\begin{theorem}
We can encode, in polynomial time, the entire set of solutions for
\ref{fip2} in a short sum of rational functions, when a finite ranking system describes the fuzzy numbers involved in the problem.
\end{theorem}
\begin{proof}
Using the $\alpha$-cuts comparation of feasible solutions, \eqref{fip2} is transformed to the family of problems \eqref{tfip2}. By hypothesis, it is enough to consider a finite set in the ranking system, so the problem is equivalent to the following multiobjective integer problem:
\begin{equation}
\label{mofip}
\begin{array}{lrl}
\max & \;  (c_1^{\alpha_1}\,x, c_2^{\alpha_1}\,x, \ldots, c_1^{\alpha_k}\,x, c_2^{\alpha_k}\,x)& \\
s.t.& x\in P \cap \mathbb{Z}_+^n \i\tag{$\mbox{\it FIP}^\alpha_{A,\tilde{c}}(b)$}
 \end{array}
 \end{equation}
where the $\alpha_1, \ldots, \alpha_k$ is the ranking system.

Then, the result follows from Theorem \ref{theo:nd}.
\end{proof}
\begin{theorem}
\label{theo:complexfip2}
If the dimension is fixed, there exists a polynomial-delay algorithm for solving \ref{fip2}.
\end{theorem}
\begin{proof}
The existence of a polynomial-delay algorithm, similar to Algorithm 1, for \ref{fip2} follows from the transformation in Theorem \ref{theo:complexfip2}.
\end{proof}
In general, fuzzy numbers are not totally described by a finite ranking system (this is the case of general LR fuzzy numbers). Then, we propose here a approximated scheme to solve these problems.
Let us consider LR-fuzzy numbers, i.e., fuzzy numbers with the following type of membership function:
$$
\mu_{\tilde{x}}(z) = \left\{\begin{array}{ll} L(\frac{a_1-z}{a_1-a_0}) & \mbox{if $z < a_1$}\\
R(\frac{z-a_1}{a_2-a_1}) & \mbox{if $z \geq a_1$}
 \end{array}\right.
 $$
where $a_0 \leq a_1 \leq a_2$ and $L$ are such that $L(0) =R(0) = 1$ and $L$ and $R$ are strictly decreasing continuous functions
on $[0,1)$. A example of this type of fuzzy numbers is shown in Figure \ref{fig:LR}.
\begin{figure}[H]
\begin{center}
\includegraphics[width=7cm, height=5.5cm]{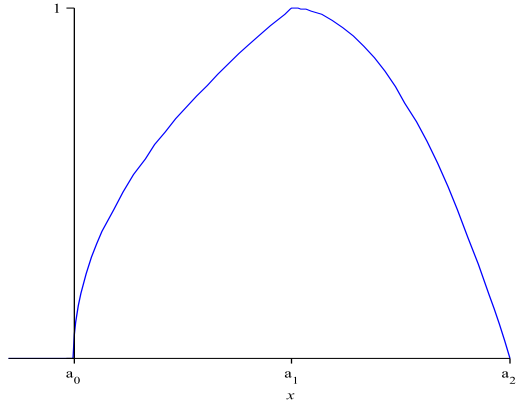}
\caption{LR membership function.\label{fig:LR}}
\end{center}
\end{figure}
Introducing as many elements as necessary in the ranking system, we can approximate the LR fuzzy number above by a continuous piecewise linear fuzzy number. Let $\alpha_i=\frac{1}{k}$ for some $k\in \N\setminus\{0\}$. Figure \ref{fig:LR_k} shows different choices for the number of elements in the system of generators (in the form $\alpha_i\frac{i}{k}$, with $k$ the number of elements).
\begin{figure}[H]
\begin{center}
\begin{tabular}{ccc}
\includegraphics[width=4cm, height=3cm]{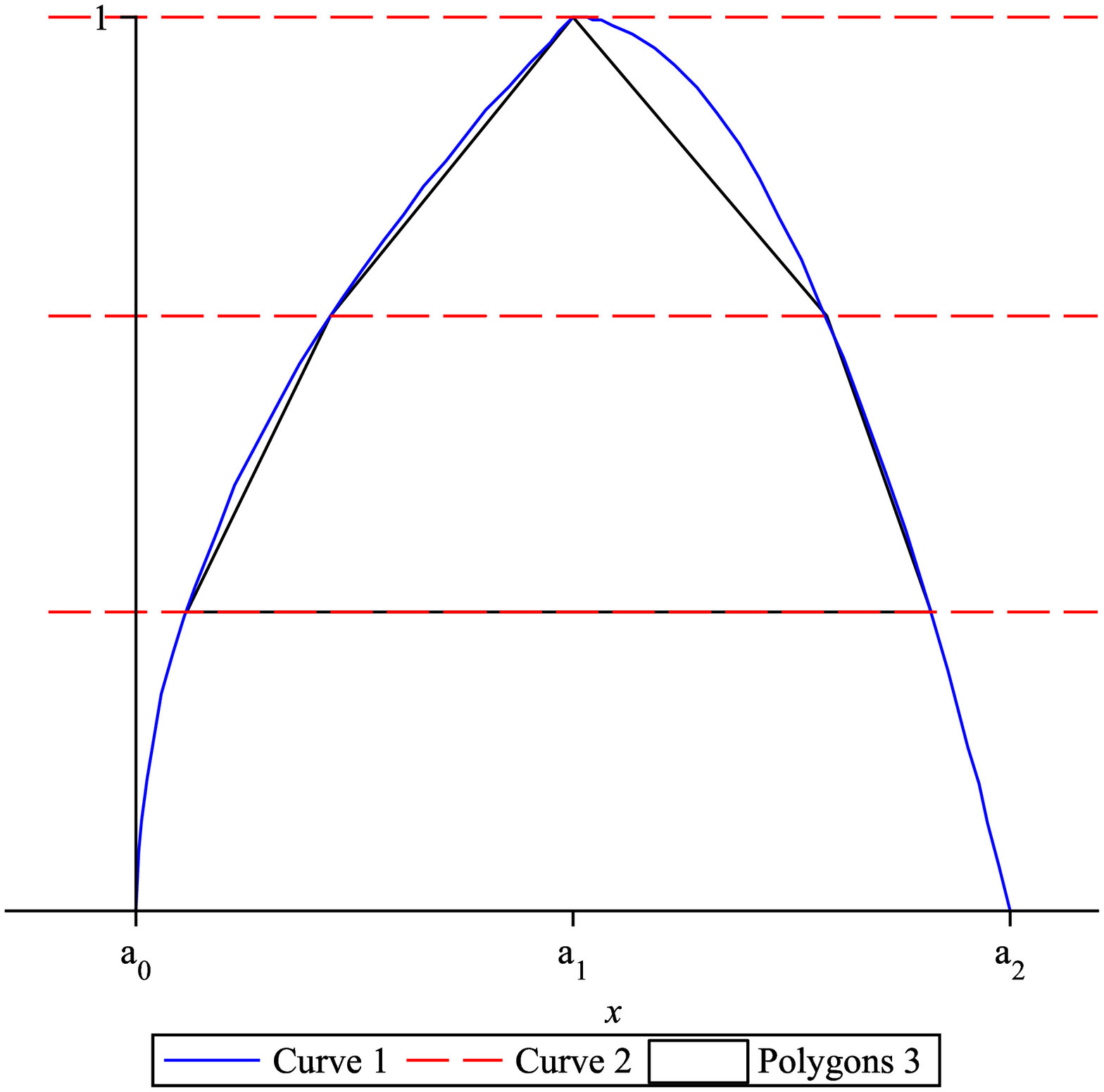} & \includegraphics[width=4cm, height=3cm]{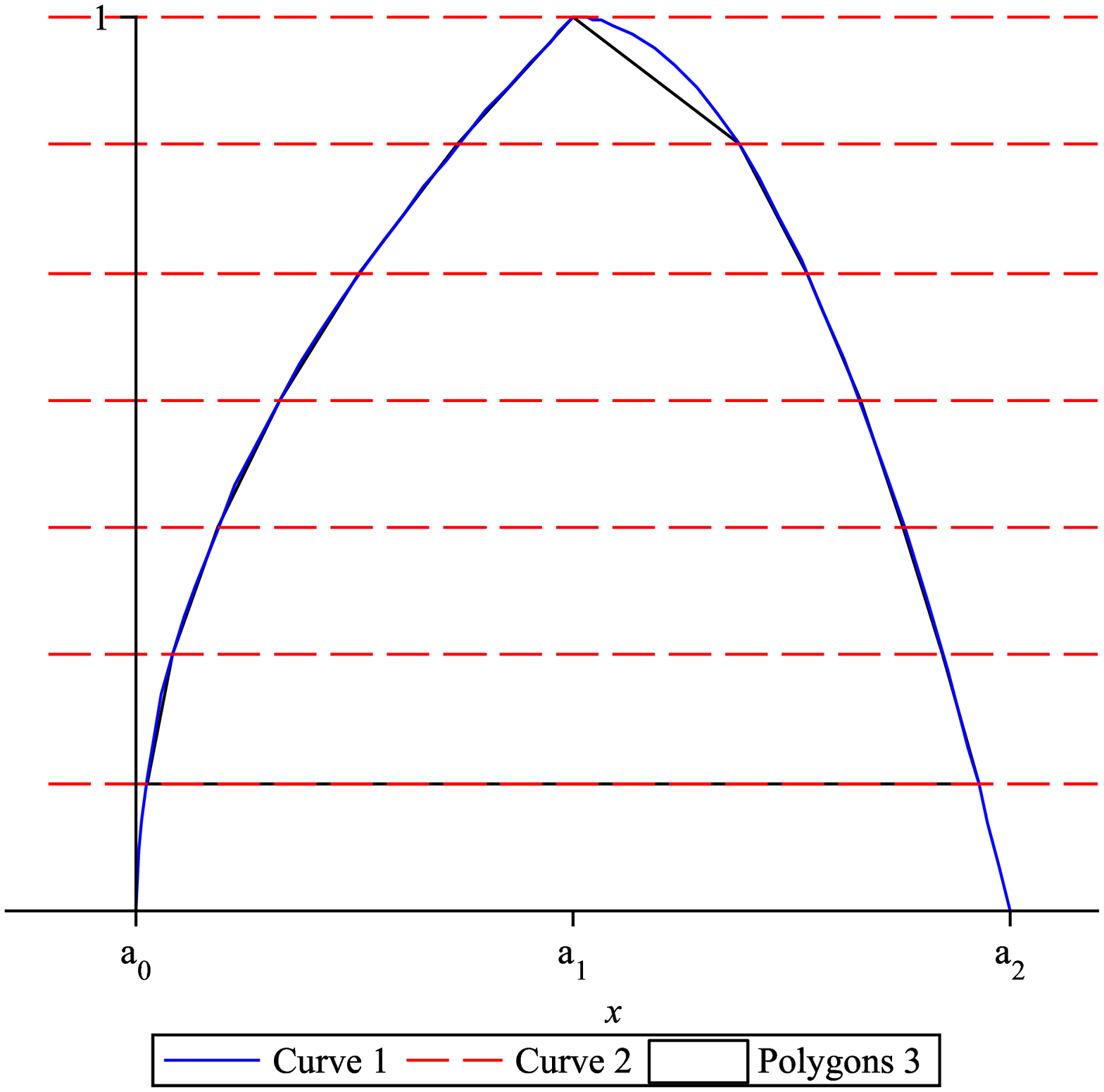} & \includegraphics[width=4cm, height=3cm]{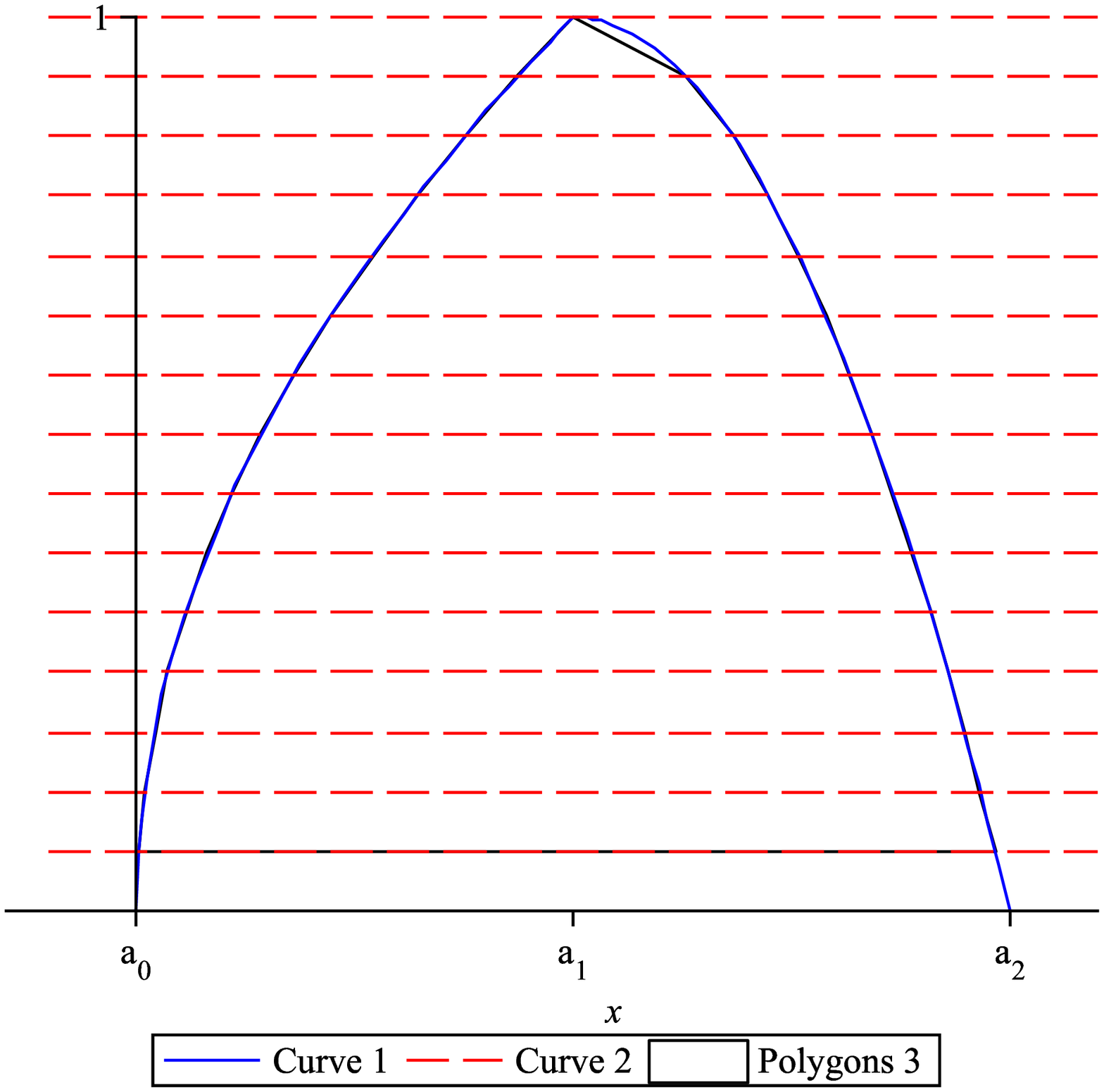}
\end{tabular}
\end{center}
\caption{Different choices for approximating a LR fuzzy number by a polygonal ($k = 3, 7, 15$).\label{fig:LR_k}}
\end{figure}
\begin{theorem}
We can encode in polynomial time, for fixed dimension, the entire set of solutions for the approximated \ref{mofip} in a short sum of rational functions.
\end{theorem}
\begin{theorem}
If the dimension is fixed, there exist a polynomial-delay algorithm for solving the approximated problem \ref{mofip}.
\end{theorem}
Note that since Theorem \ref{theo:nd} states that the complexity of multiobjective problems does not depends of the number of objective functions (provided finiteness), we can increase the number of elements in the ranking system to obtain better approximations without increasing the theoretical complexity of the problem.

The following example illustrates the methodology described above.
\begin{example}[\cite{herrera-verdegay95}]
\label{ex:2}
Consider the following problem:
\begin{equation}
\label{example2}
\begin{array}{lrl}
\max & \;  \widetilde{c}_1\,x_1+5\,x_2& \\
s.t.&  2x_1 -x_2 &\leq 12\\
&  2x_1 + 8x_2 &\leq 35\\
 & x_1, x_2&\in \mathbb{Z}_+
 \end{array}
 \end{equation}
 where $\widetilde{c}_1$ is the triangular fuzzy number given by the following membership function:
 $$
 \mu_{\widetilde{c}_1}(z) = \left\{\begin{array}{ll}
 \frac{z-1}{2} & \mbox{if $1 \leq z \leq 3$}\\
  \frac{5-z}{2} & \mbox{if $3 \leq z \leq 5$}\\
  0 & \mbox{otherwise}\end{array}\right.
    $$

    Then, the $\alpha$-cuts for the fuzzy number $\widetilde{c}\,x=\widetilde{c}_1\,x_1+5\,x_2$ are:
    $$
    (\widetilde{c}\,x)^\alpha = [(2\alpha+1)\,x_1,(5-2\alpha)\,x_1+5\,x_2]
    $$
    that define a the triangular fuzzy number given by the following membership function:
     $$
 \mu_{\widetilde{c}\,x}(z) = \left\{\begin{array}{ll}
 \frac{z-x_1+5\,x_2}{2\,x_1} & \mbox{if $x_1+5\,x_2 \leq z \leq 3\,x_1+5\,x_2$}\\
  \frac{5\,x_1+5\,x_2-z}{2\,x_1} & \mbox{if $3\,x_1+5\,x_2 \leq z \leq 5\,x_1+5\,x_2$}\\
  0 & \mbox{otherwise}\end{array}\right.
    $$
    Using the ranking system given by $\{\frac{1}{2}, 1\}$ is enough to solve the problem. After transforming our problem to a problem with 4 objective functions, we have that Problem \eqref{example2} is equivalent to:
    \begin{equation}
\label{example2_1}
\begin{array}{lrll}
\max  &(3\,x_1+5\,x_2, & 2\,x_1+5\,x_2, 3\,x_1+5\,x_2,4\,x_1+5\,x_2)\\
s.t. &2x_1 -x_2 &\leq 12&\\
&  2x_1 + 8x_2 &\leq 35&\\
 & x_1, x_2&\in \mathbb{Z}_+
 \end{array}
 \end{equation}
The entire set of nondominated solutions is $\{(4,3), (5,3), (7,2)\}$.
    \end{example}

In the following remarks we present extensions of the above problems where complexity results can be stated.

\begin{remark}
Let us consider the following fuzzy integer problem, where both constrains and objective coefficients are fuzzy numbers:
\begin{equation}
\label{fip:remark1}
\begin{array}{lrl}
\max & \;  \widetilde{c}\,x& \\
s.t.&  A\,x &\lesssim b\\
 & x&\in P \cap \mathbb{Z}_+^n
 \end{array}
 \end{equation}
Then,
\begin{enumerate}
\item If the fuzzy numbers involved in \eqref{fip:remark1} are totally described by a finite ranking system, then, the solutions of \eqref{fip:remark1} can be encoded in a short generating function in polynomial time for fixed dimension. Furthermore, those solutions can be enumerated using a polynomial delay algorithm.
\item If the fuzzy numbers involved in \eqref{fip:remark1} are not totally described by a finite ranking system, then, the solutions of an approximated modification of \eqref{fip:remark1} (with approximation error as small as desirable) can be encoded in a short generating function in polynomial time for fixed dimension. Those solutions can be enumerated using a polynomial delay algorithm.
    \end{enumerate}
\end{remark}
\begin{proof}
The result follows from the following equivalent transformation of Problem \eqref{fip:remark1}:
\begin{equation}
\begin{array}{lrl}
\max & \;  (c_1^\alpha\,x, c_2^\alpha\,x, y)& \\
s.t.& x\in P \cap \mathbb{Z}_+^n\\
 & y \in [0, M] \cap \Z
 \end{array}
 \end{equation}
 with $M$ as described for Problem \eqref{crisp3}.
 \end{proof}

\begin{remark}[Multiobjective Fuzzy Integer Programming]
Let us consider the following multiobjective fuzzy integer problem, where both constrains and objective coefficients are fuzzy numbers:
\begin{equation}
\label{fip:remark2}
\begin{array}{lrl}
\max & \;  \widetilde{C}\,x& \\
s.t.&  A\,x &\lesssim b\\
 & x&\in P \cap \mathbb{Z}_+^n
 \end{array}
 \end{equation}
 where $\widetilde{C}$ is a $k\times m$ matrix of rational fuzzy numbers.

Then,
\begin{enumerate}
\item If the fuzzy numbers involved in \eqref{fip:remark2} are totally described by a finite ranking system, then, the solutions of \eqref{fip:remark2} can be encoded in a short generating function in polynomial time for fixed dimension. Furthermore, those solutions can be enumerated using a polynomial delay algorithm.
\item If the fuzzy numbers involved in \eqref{fip:remark2} are not totally described by a finite ranking system, then, the solutions of an approximated modification of \eqref{fip:remark2} (with approximation error as small as desirable) can be encoded in a short generating function in polynomial time for fixed dimension. Those solutions can be enumerated using a polynomial delay algorithm.
    \end{enumerate}
\end{remark}
\begin{proof}
The result follows from the following equivalent transformation of Problem \eqref{fip:remark2}:
\begin{equation}
\begin{array}{lrl}
\max & \;  (c_{11}^\alpha\,x, c_{12}^\alpha\,x, \ldots, c_{k1}^\alpha\,x, c_{k2}^\alpha\,x, y)& \\
s.t.& x\in P \cap \mathbb{Z}_+^n\\
 & y \in [0, M] \cap \Z
 \end{array}
 \end{equation}
 where, $c_{j1}^\alpha$ and $c_{j2}^\alpha$ are the lower and upper extremes of the $\alpha$-cut of the $j$-th row of $\widetilde{C}$ and $M$ as described for Problem \eqref{crisp3}.
 \end{proof}

\section{Conclusions}
\label{sec:conc}
In this paper we present methodologies for solving different models of fuzzy integer programs analyzing their theoretical complexity. We deal with fuzzy integer programs with fuzzy constraints and imprecise costs. The proofs of the results presented through this paper are based on the transformations of the fuzzy problems to (crisp) multiobjective integer programs and the use of generating functions of rational polytopes. We prove new complexity results about fuzzy integer programming, concluding that: (1) Encoding the entire set of optimal solutions of a broad class of fuzzy integer programs in a short generating function is doable in polynomial time for fixed dimension; and (2) Enumerating these solutions can be done using a polynomial-delay algorithm. For problems with imprecise cost where the fuzzy numbers involved in the problem are not totally described by a finite ranking system, we present similar results but considering approximated fuzzy numbers. The advantage of the results presented for this approximation is that the theoretical complexity does not depend of the global error considered for the approximation, and then we can choose approximations as precise as we want. Finally, we give similar results also for multiobjective integer fuzzy programs.

\section{Acknowledgement}
 The authors are grateful to David Gálvez for his useful comments about the fuzzy theory. This research was partially supported by Ministerio de Educaci\'on y Ciencia under grant MTM2007-67433-C02-01 and by Junta de Andalucia under grant P06-FQM-01366.

\end{document}